\documentclass[12pt,a4paper]{article}
\usepackage{latexsym,amssymb,amsfonts,amsmath,amsthm,nccmath,float,enumitem,setspace,bm,authblk}
\usepackage[usenames,dvipsnames]{xcolor}
\usepackage[hidelinks]{hyperref}
\usepackage[margin=1.75cm]{geometry}
\usepackage{bm}
\usepackage{booktabs}
\usepackage{verbatim}
\usepackage{diagbox}
\usepackage{makecell}
\usepackage{float}
\allowdisplaybreaks[1]

\hypersetup{
	colorlinks,
	linkcolor={red!80!black},
	citecolor={blue!80!black},
	urlcolor={blue!80!black}
}

\newtheorem{thm}{Theorem}[section]

\newtheorem{pro}[thm]{Proposition}

\newtheorem{rem}[thm]{Remark}

\newtheorem{lem}[thm]{Lemma}
\newtheorem{core}[thm]{Corollary}

\newtheorem{example}{Example}[section]%

\def \leq {\leqslant}

  \def\tr{{\rm tr}}

\def\exp{{\sf exp}}

\let\oldproofname=\proofname
\renewcommand{\proofname}{\rm\bf{\oldproofname}}

\begin{document}
	
\title{Algebraic degrees of quasi-abelian semi-Cayley digraphs}
\author[a]{Shixin Wang}
\author[b]{Majid Arezoomand}
\author[a]{Tao Feng \thanks{Supported by NSFC under Grant 12271023}}
\affil[a]{School of Mathematics and Statistics, Beijing Jiaotong University, Beijing 100044, P. R. China}
\affil[b]{Department of Mathematics, University of Larestan, Lar, Iran}
\affil[ ]{sxwang@bjtu.edu.cn, arezoomand@lar.ac.ir, tfeng@bjtu.edu.cn}
	
\renewcommand*{\Affilfont}{\small\it}
\renewcommand\Authands{ and }

\date{}
\maketitle
	
\begin{abstract}
For a digraph $\Gamma$, if $F$ is the smallest field that contains all roots of the characteristic polynomial of the adjacency matrix of $\Gamma$, then $F$ is called the splitting field of $\Gamma$. The extension degree of $F$ over the field of rational numbers $\mathbb{Q}$ is said to be the algebraic degree of $\Gamma$. A digraph is a semi-Cayley digraph over a group $G$ if it admits $G$ as a semiregular automorphism group with two orbits of equal size. A semi-Cayley digraph $\mathrm{SC}(G,T_{11},T_{22},T_{12},T_{21})$ is called quasi-abelian if each of $T_{11},T_{22},T_{12}$ and $T_{21}$ is a union of some conjugacy classes of $G$. This paper determines the splitting field and the algebraic degree of a quasi-abelian semi-Cayley digraph over any finite group in terms of irreducible characters of groups. This work generalizes the previous works on algebraic degrees of Cayley graphs over abelian groups and any group having a subgroup of index 2, and semi-Cayley digraphs over abelian groups.
\end{abstract}	



\section{Introduction}

A graph is {\em integral} if all the eigenvalues of its adjacency matrix are integers. Harary and Schwenk \cite{Which Graphs Have Integral Spectra} introduced the concept of integral graphs in 1974. Since then, the classification and constructions for integral graphs have become an important research topic in algebraic graph theory. Integral graphs are very rare \cite{Graphs with integral spectrum} and in general, it is difficult to classify them.

Since the eigenvalues of graphs are all algebraic integers, M\"{o}nius, Steuding and Stumpf \cite{Which graphs have nonintegral spectra} proposed the concept of algebraic degrees of graphs in 2018. Given a graph $\Gamma$, if the field $K$ is the smallest extension of the field of rational numbers $\mathbb{Q}$ containing all the eigenvalues of $\Gamma$, then $K$ is said to be the {\em splitting field} of $\Gamma$, denoted by $\mathrm{SF}(\Gamma)$. The {\em algebraic degree} of $\Gamma$ is defined as the extension degree of $\mathrm{SF}(\Gamma)$ over $\mathbb{Q}$, i.e., $[\mathrm{SF}(\Gamma):\mathbb{Q}]$, denoted by $\mathrm{deg}(\Gamma)$. A graph is integral if and only if its algebraic degree is one.

It is shown in \cite{Which graphs have nonintegral spectra} that a graph with large diameter must have large algebraic degree. M\"{o}nius \cite{The algebraic degree of spectra of circulant graphs,Splitting fields of spectra of circulant graphs} determined the splitting field and the algebraic degree of any circulant graph. Much more work has been done, very recently, on determining algebraic degrees of mixed Cayley graphs over abelian groups \cite{Splitting fields of mixed cayley graphs over abelian groups}, Cayley hypergraphs \cite{Algebraic degree of spectra of Cayley hypergraphs}, Cayley graphs over abelian groups and dihedral groups \cite{Algebraic degree of Cayley graphs over abelian groups and dihedral groups}, and semi-Cayley digraphs over abelian groups \cite{Algebraic degrees of 2-Cayley digraphs over abelian groups}.

A {\em Cayley digraph} $\mathrm{Cay}(G,S)$ over a finite group $G$ with respect to $S$ is a digraph with vertex set $G$ and edge set $\{(x,y)\mid x,y\in G, yx^{-1}\in S\}$. Semi-Cayley graphs as a generalization of Cayley graphs were first proposed by de Resmini and Jungnickel \cite{Strongly regular semi-Cayley graphs} and also known as bi-Cayley graphs (cf. \cite{zf}) or 2-Cayley graphs (cf. \cite{On the characteristic polynomial of n-Cayley digraphs}).
A digraph is said to be a {\em semi-Cayley digraph} over a finite group $G$ if it admits $G$ as a semiregular automorphism group with two orbits of equal size.
Specifically, for four subsets $T_{ij}$, $i,j\in\{1,2\}$, of a finite group $G$, a semi-Cayley digraph $\mathrm{SC}(G,T_{11},T_{22},T_{12},T_{21})$ is a digraph with vertex set the union of the right part $G_1=\{g_1\mid g\in G\}$ and the left part $G_2=\{g_2\mid g\in G\}$, and the arc set the union of $\{(h_1,g_1)\mid gh^{-1}\in T_{11}\}$,  $\{(h_2,g_2)\mid gh^{-1}\in T_{22}\}$, $\{(h_1,g_2)\mid gh^{-1}\in T_{12}\}$ and $\{(h_2,g_1)\mid gh^{-1}\in T_{21}\}$.
Note that $\mathrm{SC}(G,T_{11},T_{22},T_{12},T_{21})$ is undirected if and only if $T_{11}=T_{11}^{-1},T_{22}=T_{22}^{-1}$ and $T_{12}^{-1}=T_{21}$.


This paper is devoted to determining the splitting fields and algebraic degrees of quasi-abelian semi-Cayley digraphs. A Cayley digraph $\mathrm{Cay}(G,S)$ is {\em quasi-abelian} if $S$ is a conjugate-closed subset of $G$ (that is, if $s\in S$ then $g^{-1}sg\in S$ for all $g\in G$), or equivalently a union of conjugacy classes of $G$ (cf. \cite{Quasi-Abelian Cayley graphs and Parsons graphs}). Similarly, a semi-Cayley digraph  $\mathrm{SC}(G,T_{11},T_{22},T_{12},T_{21})$ is said to be {\em quasi-abelian} if each of $T_{11},T_{22},T_{12}$ and $T_{21}$ is a union of conjugacy classes of $G$ (cf. \cite{Majid2}). A semi-Cayley digraph over an abelian group is a quasi-abelian semi-Cayley digraph, but not vice versa.

The paper is organized as follows. Section \ref{pre} gives some basic facts on group representations and field extensions. Section \ref{sec:3.1} provides explicit expressions for eigenvalues of quasi-abelian semi-Cayley digraphs. Section \ref{sec:3.2} determines the splitting fields and algebraic degrees of quasi-abelian semi-Cayley digraphs (see Theorem \ref{main}). Examples of non-abelian and non-Cayley semi-Cayley graphs are provided in Section \ref{sec:3.3} to illustrate the use of Theorem \ref{main}. Concluding remarks are given in Section \ref{sec:conluding}.


\section{Preliminaries}\label{pre}

All groups considered in this paper are finite and we use the multiplicative notation for the group operation. Sets and multisets will be denoted by curly braces $\{~\}$ and square brackets $[~]$, respectively. Every union will be understood as multiset union with multiplicities of elements preserved.

For multisets $X,Y\subseteq G$ and a positive integer $z$, we define the multisets $X^z:=[ x^z\mid x\in X]$, $XY:=[xy\mid x\in X, y\in Y]$ and $XY:=\emptyset$ if either $X$ or $Y$ is an empty set.
For a positive integer $k$ and a multiset $X$, $k*X$ is a multiset in which each element of $X$ appears $k$ times. For two multisets $X=\bigcup_{x\in G}a_x*\{x\}$ and $Y=\bigcup_{x\in G}b_x*\{x\}$, where $a_x$ and $b_x$ are positive integers for each $x\in G$, we define
$X\setminus Y:=\bigcup_{a_x>b_x,x\in G}(a_x-b_x)*\{x\}$ and
$Y\setminus X:=\bigcup_{a_x<b_x,x\in G}(b_x-a_x)*\{x\}$.
Clearly, $(X\setminus Y)\cap(Y\setminus X)=\emptyset$.

The reader is referred to \cite{Huppert} for representation theory of finite groups and properties of characters. Here we recall only a few basic definitions. Let $G$ be a finite group. A {\em representation of $G$ of degree $d$} is a group homomorphism $\rho:G\rightarrow \mathrm{GL}(V)$ for some  vector space $V$ of dimension $d$, where $\mathrm{GL}(V)$ is the group of all invertible linear transformations on $V$. Usually, we write $\rho_g$ instead of $\rho(g)$ and $\rho_g(v)$ is the value of $\rho_g$ on $v\in V$. Two representations $\varphi:G\rightarrow \mathrm{GL}(V)$ and $\psi:G\rightarrow \mathrm{GL}(W)$ are said to be {\em equivalent} if there exists an isomorphism $\tau:V\rightarrow W$ such that $\psi_g= \tau\varphi_g\tau^{-1}$ for all $g\in G$. For a representation $\rho:G\rightarrow \mathrm{GL}(V)$, a subspace $W\leq  V$ is said to be {\em $G$-invariant} if for all $g\in G$ and $w\in W$, one has $\rho_g(w)\in W$. The representation $\rho$ is said to be {\em irreducible} if the only $G$-invariant subspaces of $V$ are $\{0\}$ and $V$ itself. For a representation $\rho:G\rightarrow \mathrm{GL}(V)$, the {\em character} $\chi_\rho:G\rightarrow \mathbb{C}$ of $\rho$ is defined by $\chi_\rho(g)=\mathrm{Tr} (\rho_g)$, where ``$\mathrm{Tr}$'' is the trace function. For a character $\chi$ of $G$ the value $\chi(1)$ is called {\em the degree} of $\chi$ and is denoted by $d_\chi$. A character of an irreducible representation is called an {\em irreducible character}.

Let $f$ be an irreducible representation or a character of a group $G$ and $A$ is a multiset of $G$.
We define $f(A):=\sum_{a\in A}f(a)$. Denote by $\mathrm{Irr}(G)$ the set of all irreducible characters of $G$. We use the convention that $\chi(\emptyset)=0$ for any $\chi\in \mathrm{Irr}(G)$.

\begin{lem}\label{class-function}
Let $\chi\in\mathrm{Irr}(G)$ and $\rho$ be an affording irreducible representation of $\chi$. If $A$ and $B$ are two conjugate-closed subsets of $G$, then
\begin{center}
\begin{tabular}{lll}
$(1)$ $\rho(A)=\frac{\chi(A)}{d_\chi}I_{d_\chi},$ & $(2)$ $\rho(AB)=\frac{\chi(A)\chi(B)}{d_\chi^2}I_{d_\chi},$ & $(3)$ $\chi(AB)=\frac{\chi(A)\chi(B)}{d_\chi}.$
\end{tabular}
\end{center}
\end{lem}

\begin{proof} For every $g\in G$, since $\rho$ is a group homomorphism and $gAg^{-1}=A$, we have
\begin{align*}
\rho(g)\rho(A)\rho(g)^{-1}=\rho(g)\Big(\sum_{a\in A} \rho(a)\Big )\rho(g^{-1})
	=\sum_{a\in A}\rho(gag^{-1})
	=\sum_{a'\in A}\rho(a')=\rho(A).
\end{align*}
Then $(1)$ follows immediately from \cite[Theorem 2.3(b)]{Huppert}. Since $\rho$ is a group homomorphism, we have $\rho(A)\rho(B)=\rho(AB)$, and so $(2)$ is a direct consequence of $(1)$. Furthermore, $\tr(\rho(AB))=\chi(AB)$, which proves $(3)$.
\end{proof}

For a finite group $G$, the group ring $\mathbb{C}[G]$ is a ring, consisting of all the formal sums of the free basis $\{g\mid g\in G\}$ over the field of complex numbers $\mathbb{C}$, with the multiplication extending that of $G$ by linearity and distributivity. For any $\sum_{g\in G}a_gg\in \mathbb{C}[G]$ and $\chi \in \mathrm{Irr}(G)$, we define $\chi(\sum_{g\in G}a_gg):=\sum_{g\in G}a_g\chi(g)$.
The following lemma plays a significant role in the proof of Theorem \ref{main} in Section \ref{sec:main}.

\begin{lem}\label{a=b}
Let $A$ and $B$ be conjugate-closed multisets of a finite group $G$. Then $A=B$ if and only if $\chi(A)=\chi(B)$ for every $\chi\in\mathrm{Irr}(G)$.
\end{lem}

\begin{proof}
For $\chi\in\mathrm{Irr}(G)$, let
\begin{equation}\label{eqn:1}
e_\chi=\frac{\chi(1)}{|G|}\sum\limits_{x\in G}\overline{\chi(x)}x\in \mathbb{C}[G].
\end{equation}
Then by \cite[Theorem 4.5]{Huppert}, $\{e_\chi\mid \chi\in\mathrm{Irr}(G)\}$ is a $\mathbb{C}$-basis of $Z(\mathbb{C}[G])$, the center of $\mathbb{C}[G]$. For any $\mathcal{Y}\in Z(\mathbb{C}[G])$, we can write uniquely \begin{equation}\label{eqn:2}
\mathcal{Y}=\sum_{\chi\in\mathrm{Irr}(G)}a_\chi(\mathcal{Y}) e_\chi,
\end{equation}
where $a_\chi(\mathcal{Y})\in \mathbb{C}$, and so $\eta(\mathcal{Y})=\sum_{\chi\in\mathrm{Irr}(G)}a_\chi(\mathcal{Y}) \eta(e_\chi)$ for any $\eta\in \mathrm{Irr}(G)$.
By \eqref{eqn:1}, using the orthogonality relations of characters, we have
\begin{equation*}
	\eta(e_\chi)=\frac{\chi(1)}{|G|}\sum\limits_{x\in G}\overline{\chi(x)}\eta(x)=\left\{
	\begin{array}{ll}
		\chi(1),&\eta=\chi\\
		0,&\eta\neq \chi.
	\end{array}
	\right.
\end{equation*}
Thus $\eta(\mathcal{Y})=\sum_{\chi\in\mathrm{Irr}(G)}a_\chi(\mathcal{Y}) \eta(e_\chi)=a_\eta(\mathcal{Y}) \chi(1)=a_\eta(\mathcal{Y}) \eta(1)$ and so $a_\eta(\mathcal{Y})=\frac{\eta(\mathcal{Y})}{\eta(1)}$.

If $A=B$, then $\chi(A)=\chi(B)$ for all $\chi\in\mathrm{Irr}(G)$. Conversely, suppose that $\chi(A)=\chi(B)$ for all $\chi\in\mathrm{Irr}(G)$. Since $A$ and $B$ are conjugate-closed multisets of $G$, we have $\Sigma_{a\in A} a$ and $\Sigma_{b\in B} b$ are both in $Z(\mathbb{C}[G])$. It follows that  $a_\chi(A)=\frac{\chi(A)}{\chi(1)}=\frac{\chi(B)}{\chi(1)}=a_\chi(B)$ for all $\chi\in\mathrm{Irr}(G)$, which implies $A=B$ by \eqref{eqn:2}.
\end{proof}
Let $K$ be a field and $F$ be an extension field of $K$, denoted by $F/K$. The {\em degree} of the field extension $F/K$, written as $[F:K]$, is the dimension of $F$ as a vector space over $K$. The extension $F/K$ is {\em finite} if $[F:K]$ is finite and is {\em infinite} otherwise. By $\mathrm{Gal}(F/K)$, {\em the Galois group} of $F$ over $K$, we mean the set of all automorphisms of $F$ that take every element of $K$ to itself. If $F/K$ is a finite extension and the set of all elements fixed by $\mathrm{Gal}(F/K)$ in $F$ is exactly $K$, then $F/K$ is a {\em Galois extension}, and $|\mathrm{Gal}(F/K)|=[F:K]$. If $F/K$ is a finite Galois extension, then for any intermediate field $L$ such that $K\leq L\leq F$, $F/L$ is also a finite Galois extension. The reader is referred to \cite{Morandi} for Galois theory.

\begin{lem} {\rm \cite[Corollary 7.8]{Morandi}} \label{Galois}
Let $\omega_m$ be a primitive $m$-th root of unity and $\mathbb{Q}(\omega_m)$ be the cyclotomic field obtained by adjoining $\omega_m$ to $\mathbb{Q}$. Then $\mathbb{Q}(\omega_m)/\mathbb{Q}$ is a finite Galois extension. Furthermore, let $\sigma:\mathbb{Q}(\omega_m)\rightarrow \mathbb{Q}(\omega_m)$ be an element of $\mathrm{Gal}(\mathbb{Q}(\omega_m)/\mathbb{Q})$. Then $\sigma(\omega_m)=\omega_m^t$ for some $t\in\mathbb{Z}_m^*$, and
\begin{eqnarray*}
\eta:\mathrm{Gal}(\mathbb{Q}(\omega_m)/\mathbb{Q})&&\longrightarrow \mathbb{Z}_m^*,~~~~~~ \sigma \longmapsto t
\end{eqnarray*}
is a group isomorphism, where $\mathbb{Z}_m^*$ is the multiplicative group of the ring $\mathbb{Z}/m\mathbb{Z}$.
\end{lem}

For a positive integer $n$ and a field $K$, let $K^{\times}:=K\setminus \{0\}$ and $K^{\times n}=\{k^n\mid k\in K^\times \}$ be the subgroup of $K^{\times}$. The image of element $k\in K^{\times}$ under the natural epimorphism from $K^{\times}$ onto the quotient group $K^{\times}/K^{\times n}$ is denoted by $[k]_{K^{\times n}}$. When $n=2$, $[k]_{K^{\times 2}}$ is simply written as $[k]_K$.

\begin{lem} {\rm \cite[Theorem 11.4 and Proposition 11.10]{Morandi}} \label{lem:KK}
Suppose $K$ is a field containing a primitive $n$-th root of unity and let $$F=K(\sqrt[n]{a_1},\ldots,\sqrt[n]{a_r}),$$
where $a_i\in K$ for $1\leq i\leq r$. Then $\mathrm{Gal}(F/K)$ is isomorphic to the subgroup of $K^{\times}/K^{\times n}$ generated by $[a_1]_{K^{\times n}},\ldots,[a_r]_{K^{\times n}}$.
\end{lem}

\section{Algebraic degrees of quasi-abelian semi-Cayley digraphs}\label{sec:main}

\subsection{Eigenvalues of quasi-abelian semi-Cayley digraphs}\label{sec:3.1}

The algebraic degree of a graph $\Gamma$ is the extension degree $[F:\mathbb{Q}]$ where $F$ is the smallest field containing all the eigenvalues of the adjacency matrix of $\Gamma$. Therefore, to determine the algebraic degree of a quasi-abelian semi-Cayley digraph, we need to know all eigenvalues of its adjacency matrix. The following lemma can be obtained by applying \cite[Theorem 3.1]{Majid2}. Its proof is similar to that of \cite[Corollary 3.2]{Majid2} which gives eigenvalues of the adjacency matrix of a quasi-abelian semi-Cayley undirected graph. We omit its proof here.

\begin{lem}\label{lem:eigenvalues-0}
Let $\Gamma=\mathrm{SC}(G,T_{11},T_{22},T_{12},T_{21})$ be a quasi-abelian semi-Cayley digraph. Then the eigenvalues of the adjacency matrix of $\Gamma$ are
\begin{equation*}\label{eqn:eigenvalues}				\lambda_\chi^{\pm}:=\frac{\chi(T_{11})+\chi(T_{22})\pm\sqrt{(\chi(T_{11})-\chi(T_{22}))^2+4\chi(T_{12})\chi(T_{21})}}{2d_\chi},
\end{equation*}
where $\chi\in \mathrm{Irr}(G)$ and the multiplicities of $\lambda_\chi^{+}$ and $\lambda_\chi^{-}$ are both $d_\chi^2$.
\end{lem}

In this section, we always assume that
\begin{equation}\label{eqn:I}
	\begin{array}{ll}
&I_1:=T_{11}\cup T_{22},\\
&I_2:=(T_{11}T_{11})\cup (T_{22}T_{22})\cup (4*T_{12}T_{21}),\\
&I_3:=2*T_{11}T_{22}.
	\end{array}
\end{equation}
They are all conjugate-closed multisets of $G$. Using $I_1$, $I_2$ and $I_3$, we rewrite the eigenvalues in Lemma \ref{lem:eigenvalues-0} as follows.

\begin{lem}\label{lem:eigenvalues}
Let $\Gamma=\mathrm{SC}(G,T_{11},T_{22},T_{12},T_{21})$ be a quasi-abelian semi-Cayley digraph. Then the eigenvalues of the adjacency matrix of $\Gamma$ are
$$\lambda_\chi^{\pm}=\frac{\chi(I_1)\pm\sqrt{d_\chi(\chi(I_2\setminus I_3)-\chi(I_3\setminus I_2))}}{2d_\chi},$$
where $\chi\in \mathrm{Irr}(G)$ and the multiplicities of $\lambda_\chi^{+}$ and $\lambda_\chi^{-}$ are both $d_\chi^2$.
\end{lem}

\begin{proof}
Clearly, $\chi(T_{11})+\chi(T_{22})=\chi(I_1)$.
Using Lemma \ref{class-function}(3), we have
\begin{align*}
&(\chi(T_{11})-\chi(T_{22}))^2+4\chi(T_{12})\chi(T_{21})\\
=&\chi(T_{11})\chi(T_{11})-2\chi(T_{11})\chi(T_{22})+\chi(T_{22})\chi(T_{22})+4\chi(T_{12})\chi(T_{21})\\
=&d_\chi\chi(T_{11}T_{11})-2d_\chi\chi(T_{11}T_{22})+d_\chi\chi(T_{22}T_{22})+4d_\chi\chi(T_{12}T_{21})\\
=&d_\chi\chi(I_2)-d_\chi\chi(I_3)=d_\chi(\chi(I_2\setminus I_3)-\chi(I_3\setminus I_2)).
\end{align*}
Then the desired conclusion follows immediately from Lemma \ref{lem:eigenvalues-0}.
\end{proof}

\subsection{Determination of algebraic degrees}\label{sec:3.2}

By Lemma \ref{lem:eigenvalues}, the splitting filed of $\Gamma$ is
\begin{equation}\label{eqn:splitting filed 0}
\begin{split}
\mathrm{SF}(\Gamma)&=\mathbb{Q}(\chi(I_1)\pm\sqrt{d_\chi(\chi(I_2\setminus I_3)-\chi(I_3\setminus I_2))}\mid \chi\in \mathrm{Irr}(G))\\
&=\mathbb{Q}(\chi(I_1),\sqrt{d_\chi(\chi(I_2\setminus I_3)-\chi(I_3\setminus I_2))}\mid \chi\in \mathrm{Irr}(G)).
\end{split}
\end{equation}
In order to determine the algebraic degree $[\mathrm{SF}(\Gamma):\mathbb{Q}]$ of $\Gamma$, we let
\begin{equation}\label{eqn:K0}
K=\mathbb{Q}(\chi(I_1),\chi(I_2\setminus I_3)-\chi(I_3\setminus I_2)\mid \chi\in \mathrm{Irr}(G)).
\end{equation}
It follows from \eqref{eqn:splitting filed 0} that
\begin{equation}\label{splitting field}
\mathrm{SF}(\Gamma)=K(\sqrt{d_\chi(\chi(I_2\setminus I_3)-\chi(I_3\setminus I_2))}\mid \chi\in \mathrm{Irr}(G)),
\end{equation}
and so
\begin{equation}\label{eqn:splitting field 1}
\mathrm{deg}(\Gamma)=[\mathrm{SF}(\Gamma):\mathbb{Q}]=[\mathrm{SF}(\Gamma):K][K:\mathbb{Q}].
\end{equation}

\begin{lem}\label{lem:F/K}
Let $M=\langle [d_\chi(\chi(I_2\setminus I_3)-\chi(I_3\setminus I_2))]_K\mid \chi\in \mathrm{Irr}(G)\rangle$. Then $[\mathrm{SF}(\Gamma):K]=|M|$.
\end{lem}

\begin{proof}
It follows from Lemma \ref{lem:KK} (take $n=2$) that $\mathrm{Gal}(\mathrm{SF}(\Gamma)/K)$ is isomorphic to the group $M$. By Lemma \ref{Galois}, $\mathrm{SF}(\Gamma)/\mathbb{Q}$ is a finite Galois extension. Since $\mathbb{Q}\leq K\leq \mathrm{SF}(\Gamma)$, $\mathrm{SF}(\Gamma)/K$ is also a finite Galois extension. Therefore, $|\mathrm{Gal}(\mathrm{SF}(\Gamma)/K)|=[\mathrm{SF}(\Gamma):K]=|M|$.
\end{proof}

Next we determine $[K:\mathbb{Q}]$. Let $m$ be the {\em exponent} of a finite group $G$, i.e., the least common multiple of the orders of all elements of $G$. Let $\mathbb{Q}(\omega_m)$ be the cyclotomic field obtained by adjoining a primitive $m$-th root of unity $\omega_m$ to $\mathbb{Q}$. By \eqref{eqn:K0} we have $\mathbb{Q}\leq K\leq \mathbb{Q}(\omega_m)$, and so  $[K:\mathbb{Q}]=[\mathbb{Q}(\omega_m):\mathbb{Q}]/[\mathbb{Q}(\omega_m):K]$.
Since $[\mathbb{Q}(\omega_m):\mathbb{Q}]=|\mathrm{Gal}(\mathbb{Q}(\omega_m)/\mathbb{Q})|=|\mathbb{Z}_m^*|=\varphi(m)$ by Lemma \ref{Galois}, where $\varphi$ is the Euler function, in order to determine $[K:\mathbb{Q}]$ it remains to determine $[\mathbb{Q}(\omega_m):K]$. Since $\mathbb{Q}(\omega_m)/K$ is a finite Galois extension, $[\mathbb{Q}(\omega_m):K]=|\mathrm{Gal}(\mathbb{Q}(\omega_m)/K)|$. Thus we need to analyze the Galois group $\mathrm{Gal}(\mathbb{Q}(\omega_m)/K)=\{\sigma:\mathbb{Q}(\omega_m)\longrightarrow \mathbb{Q}(\omega_m) {\rm\ is\ a\ field\ isomorphism} \mid \sigma(a)=a {\rm\ for\ any}\ a\in K\}$.

By \eqref{eqn:K0}, $K=\mathbb{Q}(\chi(I_1),\chi(I_2\setminus I_3)-\chi(I_3\setminus I_2)\mid \chi\in \mathrm{Irr}(G))$, and so for each $\sigma\in \mathrm{Gal}(\mathbb Q(\omega_m)/K)$ and  $\chi\in \mathrm{Irr}(G)$, we have
\begin{equation}\label{eqn:I1}
	\sigma(\chi(I_1))=\chi(I_1)
\end{equation}
and
\begin{equation}\label{eqn:222}	
	\sigma(\chi(I_2 \setminus I_3)-\chi(I_3 \setminus I_2))=\chi(I_2\setminus I_3)-\chi(I_3 \setminus I_2).
\end{equation}
Hence to determine $\mathrm{Gal}(\mathbb{Q}(\omega_m)/K)$, we need to know when \eqref{eqn:I1} and \eqref{eqn:222} hold.

Before we go further, we make the following preparations. Let $\rho$ be an irreducible representation of $G$ affording the character $\chi$. For $g\in G$, suppose that $\alpha_1,\ldots,\alpha_{d_\chi}$ are all eigenvalues of $\rho_g$. It follows that for any positive integer $z$, the eigenvalues of $\rho_{g^z}$ are $\alpha_1^z,\ldots,\alpha_{d_\chi}^z$. Since $m$ is the exponent of $G$, $m$ is divisible by the order of $g$, so $\alpha_i$ is an $m$-th root of unity for each $i=1,\ldots,d_\chi$. Note that $\mathrm{Gal}(\mathbb{Q}(\omega_m)/K)\leq \mathrm{Gal}(\mathbb{Q}(\omega_m)/\mathbb{Q})$ and there exists a group isomorphism from $\mathrm{Gal}(\mathbb{Q}(\omega_m)/\mathbb{Q})$ to $\mathbb{Z}_m^*$ by Lemma \ref{Galois}, i.e.
\begin{equation}\label{eqn:eta}
\eta:\mathrm{Gal}(\mathbb{Q}(\omega_m)/\mathbb{Q})\longrightarrow \mathbb{Z}_m^*,~~~~~ \sigma \longmapsto t,
\end{equation}
where $t$ satisfies $\sigma(\omega_m)=\omega_m^t$.
Suppose $\sigma\in \mathrm{Gal}(\mathbb{Q}(\omega_m)/\mathbb{Q})$ and $\eta(\sigma)=t$. Then for $g\in G$,
$$\sigma(\chi(g))=\sigma(\mathrm{Tr}(\rho_g))=\sigma(\sum\limits_{i=1}^{d_\chi}\alpha_i)=
\sum\limits_{i=1}^{d_\chi}\sigma(\alpha_i)=\sum\limits_{i=1}^{d_\chi}\alpha_i^t=\chi(g^t).$$
Thus for every multiset $X$ of $G$ we have
\begin{equation}\label{eqn:Xt}
	\sigma(\chi(X))=\sigma(\sum_{g\in X}\chi(g))=\sum_{g\in X}\sigma(\chi(g))=\sum_{g\in X}\chi(g^t)=\chi(X^t).
\end{equation}
Denote by $\eta_K\subseteq \mathbb{Z}_m^*$ the image of $\mathrm{Gal}(\mathbb Q(\omega_m)/K)$ under $\eta$. Now we are ready to determine when \eqref{eqn:I1} and \eqref{eqn:222} hold.

\begin{pro}\label{pro:RLK}
$\sigma(\chi(I_1))=\chi(I_1)$ for each $\sigma\in \mathrm{Gal}(\mathbb Q(\omega_m)/K)$ and each $\chi\in \mathrm{Irr}(G)$ if and only if $I_1^t=I_1$ for all $t\in \eta_K$.
\end{pro}

\begin{proof}
By \eqref{eqn:Xt}, for $\sigma\in \mathrm{Gal}(\mathbb Q(\omega_m)/K)$ and $\chi\in \mathrm{Irr}(G)$, $\sigma(\chi(I_1))=\chi(I_1)$ if and only if $\chi(I_1^t)=\chi(I_1)$, where $t\in \eta_K$ such that $\eta(\sigma)=t$. Then applying Lemma \ref{a=b}, we have $\chi(I_1^t)=\chi(I_1)$ for each $\sigma\in \mathrm{Gal}(\mathbb Q(\omega_m)/K)$ and $\chi\in \mathrm{Irr}(G)$ if and only if $I_1^t=I_1$ for all $t\in \eta_K$.
\end{proof}

\begin{pro}\label{pro:IJ}
$\sigma(\chi(I_2 \setminus I_3)-\chi(I_3 \setminus I_2))=\chi(I_2\setminus I_3)-\chi(I_3 \setminus I_2)$ for each $\sigma\in \mathrm{Gal}(\mathbb Q(\omega_m)/K)$ and each $\chi\in \mathrm{Irr}(G)$ if and only if $I_2\setminus I_3=(I_2\setminus I_3)^t$ and $I_3\setminus I_2=(I_3\setminus I_2)^t$ for all $t\in \eta_K$.
\end{pro}

\begin{proof}
By \eqref{eqn:Xt}, for $\sigma\in \mathrm{Gal}(\mathbb Q(\omega_m)/K)$ and $\chi\in \mathrm{Irr}(G)$, $\sigma(\chi(I_2 \setminus I_3)-\chi(I_3 \setminus I_2))=\chi(I_2\setminus I_3)-\chi(I_3 \setminus I_2)$ if and only if $\chi((I_2\setminus I_3)^t)-\chi((I_3\setminus I_2)^t)=\chi(I_2\setminus I_3)-\chi(I_3 \setminus I_2)$ where $t\in \eta_K$ such that $\eta(\sigma)=t$, if and only if $\chi(I_2\setminus I_3)+\chi((I_3\setminus I_2)^t)=\chi(I_3\setminus I_2)+\chi((I_2\setminus I_3)^t)$, if and only if $\chi(A)=\chi(B)$ where $A=(I_2\setminus I_3)\cup (I_3\setminus I_2)^t$ and $B=(I_3\setminus I_2)\cup (I_2\setminus I_3)^t$. By Lemma \ref{a=b}, $\chi(A)=\chi(B)$ for all $\chi\in\mathrm{Irr}(G)$ if and only if $A=B$. Since $I_2\setminus I_3$ and $I_3\setminus I_2$ are disjoint, $A=B$ if and only if $I_2\setminus I_3=(I_2\setminus I_3)^t$ and $I_3\setminus I_2=(I_3\setminus I_2)^t$, as desired.
\end{proof}

\begin{lem}\label{pro:K}
Let $T=\{t\in \mathbb{Z}_m^*\mid (I_1)^t=I_1, (I_2\setminus I_3)^t=I_2\setminus I_3, (I_3\setminus I_2)^t=I_3\setminus I_2\}$. Then $\eta_K=T$ and $[\mathbb{Q}(\omega_m):K]=|T|$.
Moreover, $K=\{x\in \mathbb{Q}(\omega_m)\mid \sigma(x)=x \ \text{for all} \ \sigma \in \eta^{-1}(T)\}$.
\end{lem}

\begin{proof}
Since $K=\mathbb{Q}(\chi(I_1),\chi(I_2\setminus I_3)-\chi(I_3\setminus I_2)\mid \chi\in \mathrm{Irr}(G))$ by \eqref{eqn:K0}, it follows from Propositions \ref{pro:RLK} and \ref{pro:IJ} that $\eta_K=T$. Since $\eta$ is a group isomorphism, $|\mathrm{Gal}(\mathbb{Q}(\omega_m)/K)|=|\eta_K|=|T|$, which implies $[\mathbb{Q}(\omega_m):K]=|T|$. Note that $\eta^{-1}(T)=\eta^{-1}(\eta_K)=\mathrm{Gal}(\mathbb{Q}(\omega_m)/K)$ keeps every element of $K$ invariant, and hence $K=\{x\in \mathbb{Q}(\omega_m)\mid \sigma(x)=x \ \text{for all} \ \sigma \in \eta^{-1}(T)\}$.
\end{proof}

We are in a position to give our main theorem in this paper.

\begin{thm}\label{main}
Let $G$ be a finite group and $\Gamma=\mathrm{SC}(G,T_{11},T_{22},T_{12},T_{21})$ be a quasi-abelian semi-Cayley digraph. Let $m$ be the exponent of $G$ and $\omega_m$ be a primitive $m$-th root of unity. Then the splitting field of $\Gamma$ is
\begin{align*}
\mathrm{SF}(\Gamma)=&\mathbb{Q}(\chi(I_1),\sqrt{d_\chi(\chi(I_2\setminus I_3)-\chi(I_3\setminus I_2))}\mid \chi\in \mathrm{Irr}(G))\\
=&K(\sqrt{d_\chi(\chi(I_2\setminus I_3)-\chi(I_3\setminus I_2))}\mid \chi\in \mathrm{Irr}(G)),
\end{align*}
and the algebraic degree of $\Gamma$ is
$$\mathrm{deg}(\Gamma)=\frac{\varphi(m)|M|}{|T|},$$
where $I_1,I_2$ and $I_3$ are defined in \eqref{eqn:I},
$T=\{t\in \mathbb{Z}_m^*\mid (I_1)^t=I_1, (I_2\setminus I_3)^t=I_2\setminus I_3, (I_3\setminus I_2)^t=I_3\setminus I_2\}$, $K=\mathbb{Q}(\chi(I_1),\chi(I_2\setminus I_3)-\chi(I_3\setminus I_2)\mid \chi\in \mathrm{Irr}(G))=\{x\in \mathbb{Q}(\omega_m)\mid \sigma(x)=x \ \text{for all} \ \sigma \in \eta^{-1}(T)\}$, $\eta$ is defined in \eqref{eqn:eta},
and $M=\langle [d_\chi(\chi(I_2\setminus I_3)-\chi(I_3\setminus I_2))]_K\mid \chi\in \mathrm{Irr}(G)\rangle$.
\end{thm}

\begin{proof}
By \eqref{eqn:splitting filed 0}, \eqref{eqn:K0} and \eqref{splitting field}, the splitting field of $\Gamma$ is $\mathbb{Q}(\chi(I_1),\sqrt{d_\chi(\chi(I_2\setminus I_3)-\chi(I_3\setminus I_2))}\mid \chi\in \mathrm{Irr}(G))=K(\sqrt{d_\chi(\chi(I_2\setminus I_3)-\chi(I_3\setminus I_2))}\mid \chi\in \mathrm{Irr}(G))$. By \eqref{eqn:splitting field 1}, the algebraic degree of $\Gamma$ is
$$\mathrm{deg}(\Gamma)=[\mathrm{SF}(\Gamma):\mathbb{Q}]=[\mathrm{SF}(\Gamma):K][K:\mathbb{Q}]
=\frac{[\mathrm{SF}(\Gamma):K][\mathbb{Q}(\omega_m):\mathbb{Q}]}{[\mathbb{Q}(\omega_m):K]}=
\frac{|M|\varphi(m)}{|T|},$$
where the last equality follows from Lemmas \ref{Galois}, \ref{lem:F/K} and \ref{pro:K}.
\end{proof}

The following corollary gives the splitting field and the algebraic degree of an abelian semi-Cayley digraph.

\begin{core}\label{abelian} {\rm (abelian semi-Cayley digraph)}
Let $G$ be a finite abelian group and $\Gamma=\mathrm{SC}(G,T_{11}$, $T_{22},T_{12},T_{21})$. Let $m$ be the exponent of $G$ and $\omega_m$ be a primitive $m$-th root of unity.  Then the splitting field of $\Gamma$ is
$$K(\sqrt{\chi(I_2\setminus I_3)-\chi(I_3\setminus I_2)}\mid \chi\in \mathrm{Irr}(G)),$$
and its algebraic degree is
$$\mathrm{deg}(\Gamma)=\frac{\varphi(m)|M|}{|T|},$$
where $T=\{t\in \mathbb{Z}_m^*\mid (I_1)^t=I_1,(I_2\setminus I_3)^t=I_2\setminus I_3, (I_3\setminus I_2)^t=I_3\setminus I_2\}$, $K=\{x\in \mathbb{Q}(\omega_m)\mid \sigma(x)=x \ \text{for all} \ \sigma \in \eta^{-1}(T)\}$ and
$M=\langle [\chi(I_2\setminus I_3)-\chi(I_3\setminus I_2)]_K\mid \chi\in \mathrm{Irr}(G)\rangle$.
\end{core}

\begin{proof}
Since $G$ is a finite abelian group, $d_{\chi}=1$ for all $\chi\in \mathrm{Irr}(G)$. Then apply Theorem \ref{main}.
\end{proof}

\begin{rem}\label{remark:abelian}
The splitting field and the algebraic degree of a semi-Cayley digraph $\Gamma$ over an abelian group $G$ of order $n$ are also given in {\rm \cite[Theorem $3.5$]{Algebraic degrees of 2-Cayley digraphs over abelian groups}}. Let $H=\{h\in \mathbb{Z}_n^*\mid (I_1)^h=I_1,(I_2\setminus I_3)^h=I_2\setminus I_3, (I_3\setminus I_2)^h=I_3\setminus I_2\}$, $L=\{x\in \mathbb{Q}(\omega_n)\mid \sigma(x)=x \ \text{for all} \ \sigma \in \eta^{-1}(H)\}$ and
$M=\langle [\chi(I_2\setminus I_3)-\chi(I_3\setminus I_2)]_K\mid \chi\in \mathrm{Irr}(G)\rangle$. It is shown in {\rm \cite[Theorem $3.5$]{Algebraic degrees of 2-Cayley digraphs over abelian groups}} that the splitting field of $\Gamma$ is $L(\sqrt{\chi(I_2\setminus I_3)-\chi(I_3\setminus I_2)}\mid \chi\in \mathrm{Irr}(G))$ and its algebraic degree is $\mathrm{deg}(\Gamma)=\frac{\varphi(n)|M|}{|H|}$. We can apply Corollary $\ref{abelian}$ to obtain {\rm \cite[Theorem $3.5$]{Algebraic degrees of 2-Cayley digraphs over abelian groups}} by using the one-to-one mapping $\phi$ from $\mathbb{Z}_n^*/H$ to $\mathbb{Z}_m^*/T$ such that $\phi(aH)=(a \pmod{m})T$ for $aH\in \mathbb{Z}_n^*/H$, which implies $\frac{\varphi(m)}{|T|}=\frac{\varphi(n)}{|H|}$. When the exponent $m$ of $G$ is less than the order $n$ of $G$, especially when $G$ is a $p$-group, it is easier to apply Corollary $\ref{abelian}$ to count $T$ and $K$ than to apply {\rm \cite[Theorem $3.5$]{Algebraic degrees of 2-Cayley digraphs over abelian groups}} to count $H$ and $L$.
\end{rem}

\begin{core}\label{core:Cayley} {\rm (quasi-abelian Cayley digraph)}
Let $G$ be a finite group and $m$ be the exponent of $G$. Let $\Gamma=\mathrm{Cay}(G,S)$ be a quasi-abelian Cayley digraph. Then the splitting field of $\Gamma$ is
$$\mathrm{SF}(\Gamma)=\{x\in \mathbb{Q}(\omega_m)\mid \sigma(x)=x \ \text{for all} \ \sigma \in \eta^{-1}(T)\}$$
and its algebraic degree is
$$\mathrm{deg}(\Gamma)=\frac{\varphi(m)}{|T|},$$
where $T=\{t\in \mathbb{Z}_m^*\mid S^t=S\}$ and $\eta$ is defined in \eqref{eqn:eta}.
\end{core}
	
\begin{proof}
It is clear that $\mathrm{Cay}(G,S)$ and $\mathrm{SC}(G,S,\emptyset,\emptyset,\emptyset)$ have the same splitting field, and so have the same algebraic degree. Then $I_1=S$, $I_2=S^2$ and $I_3=\emptyset$. By Lemma \ref{class-function}(3), $d_\chi\chi(S^2)=(\chi(S))^2$. Thus $K=\mathbb{Q}(\chi(I_1),\chi(I_2\setminus I_3)-\chi(I_3\setminus I_2)\mid \chi\in \mathrm{Irr}(G))=\mathbb{Q}(\chi(S)\mid \chi\in \mathrm{Irr}(G))$. It follows from Theorem \ref{main} that the splitting field of $\Gamma$ is $K(\sqrt{d_\chi(\chi(I_2\setminus I_3)-\chi(I_3\setminus I_2))}\mid \chi\in \mathrm{Irr}(G))=K(\sqrt{(\chi(S))^2}\mid \chi\in \mathrm{Irr}(G))=K=\{x\in \mathbb{Q}(\omega_m)\mid \sigma(x)=x \ \text{for all} \ \sigma \in \eta^{-1}(T)\}$, $M=\langle [(\chi(S))^2]_K\mid \chi\in \mathrm{Irr}(G)\rangle=\{[1]_K\}$ and the algebraic degree of $\Gamma$ is $\frac{\varphi(m)|M|}{|T|}=\frac{\varphi(m)}{|T|}$.
\end{proof}

\begin{rem}
Corollary $\ref{core:Cayley}$ is a generalization of {\rm \cite[Theorem 1]{Algebraic degree of Cayley graphs over abelian groups and dihedral groups}} that determines the splitting field and the algebraic degree of a Cayley digraph over a finite abelian group.
\end{rem}

Let $G$ be a finite group and $S\subseteq G$. $\mathrm{BCay}(G,S)$ is known as a graph with vertex set $G_1\cup G_2$ and edge set $\{(g_1,(sg)_2)\mid g\in G, s\in S\}$ (cf. \cite{Bcay}). The following corollary gives the splitting field and the algebraic degree of $\mathrm{BCay}(G,S)$ when $S$ is conjugate-closed.

\begin{core}
Let $G$ be a finite group, $m$ be the exponent of $G$, $S$ be a conjugate-closed subset of $G$ and $\Gamma=\mathrm{BCay}(G,S)$. Then the splitting field of $\Gamma$ is
$$K(|\chi(S)|\mid \chi\in \mathrm{Irr}(G)),$$
and the algebraic degree of $\Gamma$ is
$$\mathrm{deg}(\Gamma)=\frac{\varphi(m)|M|}{|T|},$$
where $T=\{t\in \mathbb{Z}_m^*\mid (SS^{-1})^t=SS^{-1}\}$, $K=\{x\in \mathbb{Q}(\omega_m)\mid \sigma(x)=x \ \text{for all} \ \sigma \in \eta^{-1}(T)\}$ and $M=\langle [|\chi(S)|^2]_K\mid \chi\in \mathrm{Irr}(G)\rangle$.
\end{core}

\begin{proof}
Clearly $\Gamma=\mathrm{SC}(G,\emptyset,\emptyset,S,S^{-1})$. Then $I_1=I_3=\emptyset$ and $I_2=4*SS^{-1}$. By Theorem \ref{main} and Lemma \ref{class-function}(3), the splitting field of $\Gamma$ is $K(\sqrt{d_\chi\chi(I_2)}\mid \chi\in \mathrm{Irr}(G))=K(\sqrt{d_\chi\chi(SS^{-1})}\mid \chi\in \mathrm{Irr}(G))=K(\sqrt{\chi(S)\chi(S^{-1})}\mid \chi\in \mathrm{Irr}(G))=K(|\chi(S)|\mid \chi\in \mathrm{Irr}(G))$, and the algebraic degree of $\Gamma$ is $\mathrm{deg}(\Gamma)=\frac{\varphi(m)|M|}{|T|}$.
\end{proof}

\begin{core}
Let $G$ be a finite group and $m$ be the exponent of $G$. Let $\Gamma=\mathrm{SC}(G,T_{11},T_{22},T_{12},T_{21})$ be a quasi-abelian semi-Cayley digraph with $T_{11}=T_{22}$ and $T_{12}=T_{12}^{-1}=T_{21}$. Then the splitting field of $\Gamma$ is $$\mathrm{SF}(\Gamma)=\mathbb{Q}(\chi(T_{11}),\chi(T_{12})\mid \chi\in \mathrm{Irr}(G))=K(\chi(T_{12})\mid \chi\in \mathrm{Irr}(G))$$
and the algebraic degree of $\Gamma$ is
$$\mathrm{deg}(\Gamma)=\frac{\varphi(m)|M|}{|T|},$$
where $T=\{t\in \mathbb{Z}_m^*\mid T_{11}^t=T_{11}, (T_{12}T_{12})^t=T_{12}T_{12}\}$, $K=\{x\in \mathbb{Q}(\omega_m)\mid \sigma(x)=x \ \text{for all} \ \sigma \in \eta^{-1}(T)\}$ and $M=\langle [(\chi(T_{12}))^2]_K\mid \chi\in \mathrm{Irr}(G)\rangle$.
\end{core}

\begin{proof}
Note that $I_1=2*T_{11}$,  $I_2\setminus I_3=4*(T_{12}T_{12}^{-1})$ and $I_3\setminus I_2=\emptyset$. We have
$$\chi(I_2\setminus I_3)-\chi(I_3\setminus I_2)=4\chi(T_{12}T_{12}^{-1})=\frac{4|\chi(T_{12})|^2}{d_\chi}=\frac{4\chi(T_{12})^2}{d_\chi},$$
where the second equality follows from Lemma \ref{class-function}(3) and the third equality follows from the fact that $T_{12}=T_{21}^{-1}$. It follows from Theorem \ref{main} that the splitting field of $\Gamma$ is $\mathbb{Q}(\chi(T_{11}),\chi(T_{12})\mid \chi\in \mathrm{Irr}(G))=K(\chi(T_{12})\mid \chi\in \mathrm{Irr}(G))$, and the algebraic degree of $\Gamma$ is $\frac{\varphi(m)|M|}{|T|}$.
\end{proof}



\begin{core}\label{core:integral}{\rm (integral quasi-abelian semi-Cayley digraph)}
Let $G$ be a finite group and $m$ be the exponent of $G$. A quasi-abelian semi-Cayley digraph $\Gamma=\mathrm{SC}(G,T_{11},T_{22},T_{12},T_{21})$ is integral if and only if $T=\mathbb{Z}_m^*$ and $M=1$, where $T$ and $M$ are given in Theorem {\rm \ref{main}}.
\end{core}

\begin{proof}
$\Gamma$ is an integral graph if and only if $\mathrm{deg}(\Gamma)=1$, that is to say $\frac{\varphi(m)}{|T|}|M|=1$ by Theorem \ref{main}.
Since $T$ is a subgroup of $\mathbb{Z}_m^*$ and $|\mathbb{Z}_m^*|=\varphi(m)$, $\frac{\varphi(m)}{|T|}$ is an integer, so $|\mathbb{Z}_m^*/T|=1$ and $|M|=1$. Thus $\Gamma$ is an integral graph if and only if $T=\mathbb{Z}_m^*$ and $M=1$.
\end{proof}

\subsection{Examples}\label{sec:3.3}

In this section we construct some examples to illustrate the use of Theorem \ref{main}. It is readily checked that Examples \ref{ex:1}, \ref{ex:3} and \ref{ex:4} are non-abelian semi-Cayley graphs which are non-Cayley. For $g\in G$, we use the notation $g^G:=\{h^{-1}gh\mid h\in G\}$.


\begin{example}\label{ex:1}
Let $G=S_3$, $T_{11}=T_{12}=T_{21}=(12)^G$, $T_{22}=(123)^G$ and $\Gamma=\mathrm{SC}(G,T_{11},T_{22},T_{12},T_{21})$. Then $I_1=S_3\setminus\{1_G\}$, $I_2=[17*1_G,16*(123),16*(132)]$ and $I_3=[4*(13),4*(23),4*(12)]$. So $I_2\setminus I_3=I_2$ and $I_3\setminus I_2=I_3$. The exponent of $G$ is $m=6$ and $T=\{1,5\}=\mathbb{Z}_6^*$. The character table of $S_3$ is listed in Table {\rm \ref{table:s3}}.
\begin{table}[H]\centering
	\begin{tabular}{c|cccc}\toprule
		$g$ & $1_G$ & $(12)$ & $(123)$ \\
        $|g^G|$ & $1$ & $3$ & $2$ \\ \hline
		$\chi_1$&$1$&$1$&1 \\
		$\chi_2$&$1$&$-1$&1 \\
		$\chi_3$&$2$&$0$&$-1$ \\
		\bottomrule
	\end{tabular}\caption{Character table of $S_3$}\label{table:s3}
\end{table}
\noindent Then
\begin{center}
\begin{tabular}{llll}
    $\chi_1(I_1)=5$, & $\chi_1(I_2)=49$, & $\chi_1(I_3)=12$;\\
	$\chi_2(I_1)=1$, & $\chi_2(I_2)=49$, & $\chi_2(I_3)=-12$;\\
	$\chi_3(I_1)=-2$, & $\chi_3(I_2)=2$, & $\chi_3(I_3)=0$.
\end{tabular}
\end{center}
Thus $K=\mathbb{Q}(\chi(I_1),\chi(I_2\setminus I_3)-\chi(I_3\setminus I_2)\mid \chi\in \mathrm{Irr}(G))=\mathbb{Q}$, and so by Theorem {\rm \ref{main}}, the splitting field of $\Gamma$ is $K(\sqrt{d_\chi(\chi(I_2\setminus I_3)-\chi(I_3\setminus I_2))}\mid \chi\in \mathrm{Irr}(G))=\mathbb{Q}(\sqrt{37},\sqrt{61},\sqrt{4})=\mathbb{Q}(\sqrt{37},\sqrt{61})$. Since $M=\langle [37]_\mathbb{Q},[61]_\mathbb{Q},[4]_\mathbb{Q}\rangle=\langle [37]_\mathbb{Q},[61]_\mathbb{Q}\rangle$, we have $\deg(\Gamma)=\frac{\varphi(6)|M|}{|T|}=|M|=4$.

\end{example}

\begin{example}\label{ex:3}
Let $G=\langle a,b\rangle\cong A_4$, where $a=(12)(34)$ and $b=(123)$. Let $T_{11}=T_{22}=b^G$, $T_{21}=\{1_G\}$, $T_{12}=(b^{-1})^G$ and $\Gamma=\mathrm{SC}(G,T_{11},T_{22},T_{12},T_{21})$. Then $I_1=[2*(243),2*(123),2*(134),2*(142)]$, $I_2=[12*(234),12*(124),12*(132),12*(143)]$ and  $I_3=[8*(234),8*(124),8*(132),8*(143)]$. So $I_2\setminus I_3=[4*(234),4*(124),4*(132),4*(143)]$ and $I_3\setminus I_2=\emptyset$. The exponent of $G$ is $m=6$ and $T=\{1_G\}$. The character table of $A_4$ is listed in Table {\rm \ref{table:a4}}, where $\omega=\exp(\frac{2\pi\mathrm{i}}{3})=-\frac{1}{2}+\mathrm{i}\frac{\sqrt{3}}{2}$.
\begin{table}[H]\centering
	\begin{tabular}{c|cccc}\toprule
		$g$&$1_G$&$(12)(34)$&$(123)$& $(132)$\\
        $|g^G|$& 1 & 3 & 4 & 4 \\ \hline
		$\chi_1$&$1$&$1$&1&1\\
		$\chi_2$&$1$&$1$&$\omega$&$\omega^2$\\
		$\chi_3$&$1$&$1$&$\omega^2$&$\omega$\\
		$\chi_4$&$3$&$-1$&0&0\\
		\bottomrule
	\end{tabular}\caption{Character table of $A_4$}\label{table:a4}
\end{table}
\noindent Then
\begin{center}
\begin{tabular}{llll}
	$\chi_1(I_1)=8$, & $\chi_1(I_2\setminus I_3)=16$;\\
	$\chi_2(I_1)=8\omega$, & $\chi_2(I_2\setminus I_3)=16\omega^2$;\\
	$\chi_3(I_1)=8\omega^2$,& $\chi_3(I_2\setminus I_3)=16\omega$;\\
	$\chi_4(I_1)=0$, & $\chi_4(I_2\setminus I_3)=0$.
\end{tabular}
\end{center}
This yields that $K=\mathbb{Q}(\omega)$, $\mathrm{SF}(\Gamma)=K$ and $M=1$. Hence $\deg(\Gamma)=2$.

\end{example}

\begin{example}\label{ex:4}
Let $n>1$ be an odd integer and $G=D_{2n}=\langle a,b\mid a^n=b^2=(ab)^2=1\rangle$. Let $T_{11}=\langle a\rangle\setminus\{1\}$, $T_{22}=b\langle a\rangle$, $T_{12}=T_{21}=\{1\}$, and $\Gamma=\mathrm{SC}(G,T_{11},T_{22},T_{12},T_{21})$. Then $I_1=\{a^i, b,ba^i\mid 1\leq i\leq n-1\}$, $I_2=\{(2n+3)*1,n*a^i\mid 1\leq i\leq n-1\}$ and $I_3=\{2(n-1)*b,2(n-1)*ba^i\mid 1\leq i\leq n-1\}$.
Since $I_2$ and $I_3$ are disjoint, $I_2\setminus I_3=I_2$ and $I_3\setminus I_2=I_3$.  The exponent of $G$ is $m=n$ and $T=\mathbb{Z}_n^*$. The character table of $G$ is listed in Table {\rm \ref{table:2}}.
\begin{table}[H]\centering
	\begin{tabular}{c|cc}\toprule
		&$a^k~(0\leq k\leq n-1)$&$ba^k~(0\leq k\leq n-1)$\\ \hline
		$\chi_1$&$1$&$1$\\
		$\chi_2$&$1$&$-1$\\
		$\chi_l~(3\leq l\leq 2+\frac{n-1}{2})$&$2\cos(\frac{2(l-2)k\pi}{n})$&$0$\\
		\bottomrule
	\end{tabular}\caption{Character table of $D_{2n}$ with odd $n$}\label{table:2}
\end{table}
\noindent Then
\begin{center}
\begin{tabular}{llll}
$\chi_1(I_1)=2n-1$,&$\chi_1(I_2)=2n^2-2n+5$,&$\chi_1(I_3)=2n(n-1)$;\\
$\chi_2(I_1)=-1$,&$\chi_2(I_2)=2n^2-2n+5$,&$\chi_2(I_3)=-2n(n-1)$,\\
\end{tabular}
\end{center}
and for $3\leq l\leq 2+\frac{n-1}{2}$,
\begin{center}
\begin{tabular}{llll}
	$\chi_l(I_1)=\sum\limits_{k=1}^{n-1}2\cos(\frac{2(l-2)k\pi}{n})=-2$,\\
	$\chi_l(I_2)=2(2n+3)+(2n-2)\sum\limits_{k=1}^{n-1}2\cos(\frac{2(l-2)k\pi}{n})=10$,\\
    $\chi_l(I_3)=0$.
\end{tabular}
\end{center}
This shows that $K=\mathbb{Q}$, $\mathrm{SF}(\Gamma)=\mathbb{Q}(\sqrt{5},\sqrt{4n^2-4n+5},\sqrt{20})=\mathbb{Q}(\sqrt{5},\sqrt{4n^2-4n+5})$ and $M=\langle [5]_{\mathbb{Q}},[4n^2-4n+5]_{\mathbb{Q}},[20]_{\mathbb{Q}}\rangle=\langle [5]_{\mathbb{Q}},[4n^2-4n+5]_{\mathbb{Q}}\rangle$.
Thus $\mathrm{deg}(\Gamma)=|M|$. Take $n=5$ for example. We have $\mathrm{SF}(\Gamma)=\mathbb{Q}(\sqrt{5},\sqrt{85})$ and $\mathrm{deg}(\Gamma)=|M|=|\langle [5]_{\mathbb{Q}},[85]_{\mathbb{Q}}\rangle|=4$.

\end{example}

\section{Concluding remarks}\label{sec:conluding}

Focusing on quasi-abelian semi-Cayley digraphs, this paper initializes the study of determining splitting fields and algebraic degrees of semi-Cayley digraphs over non-abelian groups (see Theorem \ref{main}). Corollary \ref{core:integral} gives a sufficient and necessary condition for a  quasi-abelian semi-Cayley digraph to be integral, but classifying all integral quasi-abelian semi-Cayley digraphs is still a challenging problem. Integral quasi-abelian semi-Cayley graphs were recently proposed in \cite{wat} as a model for quantum spin networks that permit a quantum phenomenon called perfect state transfer. This led us to study the algebraic degrees of quasi-abelian semi-Cayley graphs.

It is clear that a subset of a group $G$ is conjugate-closed if and only if every inner automorphism of $G$ fixes it setwise. So we encourage the reader to study the following question. Let $S$ be a subset of $G$ which is fixed setwise by any automorphism of $G$. Determine the splitting field and algebraic degree of $\mathrm{Cay}(G,S)$. More generally, let $T_{ij}\subseteq G$, $i,j\in\{0,1\}$, each of which is fixd setwise by any automorphism of $G$. Determine the splitting field and algebraic degree of $\mathrm{SC}(G,T_{11},T_{22},T_{12},T_{21})$.

\section*{Acknowledgments}

Research of this paper was carried out while the second author was visiting Beijing Jiaotong University. He expresses his sincere thanks to the 111 Project of China (grant number B16002) for financial support and to the School of Mathematics and Statistics of Beijing Jiaotong University for their kind hospitality.

\end{document}